\documentclass[11pt]{amsart}
\usepackage{amsmath, amsfonts, amssymb, amstext, amscd, amsthm}
\numberwithin{equation}{section}
\allowdisplaybreaks

\newtheorem{torsion}{Theorem}
\newtheorem{exhaust}[torsion]{Proposition}
\newtheorem{common-sum}[torsion]{Proposition}
\newtheorem{magic-square}[torsion]{Theorem}
\newtheorem{loubere}[torsion]{Corollary}

\begin{document}

\title{Semi-Magic Squares and Elliptic Curves}


\author{Edray Herber Goins}
\address{Purdue University \\ Department of Mathematics \\ Mathematical Sciences Building \\ 150 North University Street \\ West Lafayette, IN 47907-2067}
\email{egoins@math.purdue.edu}

\begin{abstract}  We show that, for all odd natural numbers $N$, the $N$-torsion points on an elliptic curve may be placed in an $N \times N$ grid such that the sum of each column and each row is the point at infinity. \end{abstract}


\maketitle

\section{Introduction}

Let $N$ be a positive integer, and consider the integers 1, 2, \dots, $N^2$.  An $N \times N$ grid containing these consecutive integers such that the sum of each column and each row is the same is called a magic square.  (This is usually called a semi-magic square in the literature; see \cite{Weis1}.)  For example, when $N = 3$, we have the grid
\[ \begin{tabular}{|c|c|c|} \hline 3 & 5 & 7 \\ \hline 8 & 1 & 6 \\ \hline 4 & 9 & 2 \\ \hline \end{tabular} \]
\noindent where the sum of each column and each row is 15.

We need not limit ourselves to a grid with integer entries.  The author of \cite{MR1860304}, inspired by the discussion in \cite[\S 1.4]{MR1469740}, considered the problem of arranging the 9 points of inflection on an elliptic curve in a $3 \times 3$ magic square.  That is, it is possible to arrange the points of order 3 in a $3 \times 3$ grid so that the sum of each row and each column is the same, namely the point at infinity.  We generalize this result:
\begin{torsion}  \label{torsion} Let $N \geq 1$ be an odd integer, let $E$ be an elliptic curve defined over an algebraically closed field with characteristic not dividing $N$.  Then the $N^2$ points of order $N$ on $E$ can be placed in an $N \times N$ magic square such that the sum of each column and each row is the point at infinity $\mathcal O$. \end{torsion}
\noindent We construct such a grid using Lehmer's Uniform Step Method, as motivated by the discussion in \cite{MR514402}.  In particular, the theorem holds for any group $G$ such that the $N$-torsion $G[N] \simeq \left( \mathbb Z / N \mathbb Z \right) \times \left( \mathbb Z / N \mathbb Z \right)$.

The author would like to thank Jiu-Kang Yu for helpful comments.

\section{Semi-Magic Squares over Abelian Groups}

As stated above, we define a magic square to be an $N \times N$ grid containing the consecutive integers 1 through $N^2$ such that the sum of each column and each row is the same.  Strictly speaking, this is a semi-magic square, but we abuse notation slightly for the sake of brevity.  We do not limit ourselves to constructing magic squares with integer entries.  Indeed, we will construct an $N \times N$ magic square for a certain class of abelian groups.

Let $G$ be an abelian group under $\oplus$.  Given $P \in G$, denote $[-1]P$ as its inverse and $[0] P = \mathcal O$ as the identity.  For each nonzero integer $m$, denote $[m]P$ as $[\pm 1]P$ added to itself $|m|$ times, where ``$\pm$'' is chosen as the sign of $m$.  Denote $G[m] \subseteq G$ as that subgroup consisting of points $P \in G$ such that $[m]P = \mathcal O$.  We will always assume that $G$ is chosen such that for some positive integer $N$ there is a group isomorphism
\begin{equation} \label{isomorphism} \psi: \qquad \begin{CD} \left( \mathbb Z / N \mathbb Z \right) \times \left( \mathbb Z / N \mathbb Z \right) @>{\sim}>> G[N]. \end{CD} \end{equation}
\noindent We have a bijection $\{ 1, 2, \dots, N^2 \bigr \} \to \left( \mathbb Z / N \mathbb Z \right) \times \left( \mathbb Z / N \mathbb Z \right)$ given by
\begin{equation} \label{bijection} \phi: \qquad k \mapsto \left( k-1 \pmod{N}, \quad \left \lfloor \frac {k-1}N \right \rfloor \pmod{n} \right) \end{equation}
\noindent where $\lfloor \cdot \rfloor$ is the greatest integer function.  That is, if $1 \leq k \leq N^2$ then we can write $k-1 = m + N \, n$ for some unique $0 \leq m, \, n < N$, and so we map $k \mapsto (m,n)$.  This means we have a bijection of sets
\[ \psi \circ \phi: \qquad \begin{CD} \{ 1, 2, \dots, N^2 \bigr \} @>{\sim}>> G[N]. \end{CD} \]
\noindent We will use this identification to place the elements in $G[N]$ in an $N \times N$ magic square.

There are two examples in particular which will be of interest to us.  Upon fixing $N$, the group $G = \left( \mathbb Z / N \mathbb Z \right) \times \left( \mathbb Z / N \mathbb Z \right)$ satisfies the criterion above.  As another example, fix an algebraically closed field $F$ and let $E$ be an elliptic curve defined over $F$.  We may choose $G = E(F)$ as the $F$-rational points on $E$, where we have a non-canonical isomorphism $G[N] \simeq \left( \mathbb Z / N \mathbb Z \right) \times \left( \mathbb Z / N \mathbb Z \right)$ only when the characteristic of $F$ does not divide $N$.  (For more properties of elliptic curves, see \cite{MR87g:11070}.)

\section{Uniform Step Method}

Fix a positive integer $N$.  Let $G$ be an abelian group under $\oplus$, and assume
\[ G[N] = \bigl \{ R_1, \, R_2, \, \dots, \, R_k, \, \dots, \, R_{N^2} \bigr \} \simeq \left( \mathbb Z / N \mathbb Z \right) \times \left( \mathbb Z / N \mathbb Z \right). \]
\noindent We wish to place these $N^2$ elements in an $N \times N$ grid such that the sum of each row and the sum of each column, as an element in $G$, is the same.  We use an idea of D. H. Lehmer from 1929, known as the Uniform Step Method.  To this end, we are motivated by the discussion in \cite[Chapter 4]{MR514402}.

Given an $N \times N$ grid, we consider its entries in cartesian coordinates.  For the moment, fix integers $a$, $b$, $c$, and $d$, and consider placing the element $R_k \in G[N]$ in the $(x_k, y_k)$ position.  After arbitrarily placing $R_1$ in the $(x_1, y_1)$-position, we will define $x_k$ and $y_k$ by the recursive sequence
\[ \begin{aligned} x_k & \equiv x_1 + a \, (k-1) + b \left \lfloor \frac {k-1}N \right \rfloor \pmod {N} \\ y_k & \equiv y_1 + c \, (k-1) + d \left \lfloor \frac {k-1}N \right \rfloor \pmod{N}  \end{aligned} \qquad \text{for} \qquad 1 \leq k \leq N^2. \]
\noindent We will exhibit conditions on these integers $a$, $b$, $c$, and $d$ such that the sequences above indeed generate a magic square.

\begin{exhaust}  If $N$ is odd and relatively prime to $\left(a \, d-b \, c \right)$, then the sequence $(x_k, y_k)$ places exactly one $R_k$ in each of the $N^2$ cells of the $N \times N$ grid.
\end{exhaust}

\begin{proof}  It suffices to show that $(x_{k_1}, y_{k_1}) = (x_{k_2}, y_{k_2})$ only when $k_1 = k_2$; for then we would have $N^2$ different points so they must fill in the entire grid.  Using the bijection $\phi$ as in \eqref{bijection} note that we may write
\[ \begin{aligned} x_k & \equiv x_1 + a \, m + b \, n \pmod{N} \\ y_k & \equiv y_1 + c \, m + d \, n \pmod{N} \end{aligned} \qquad \text{where} \qquad (m,n) = \phi(k). \]
\noindent Write $(m_1,n_1) = \phi(k_1)$ and $(m_2,n_2) = \phi(k_2)$, so that
\[ (x_{k_1}, y_{k_1}) = (x_{k_1}, y_{k_2}) \quad \iff \quad \begin{aligned} a \, (m_1 - m_2) + b \, (n_1 - n_2) & \equiv 0 \pmod{N} \\ c \, (m_1 - m_2) + d \, (n_1 - n_2) & \equiv 0 \pmod{N} \end{aligned} \]
\noindent Since $a d - b c \pmod{N}$ is invertible, we see that this happens if and only if
\[ \phi(k_1) = \left( m_1, \ n_1 \right) = \left( m_2, \ n_2 \right) = \phi(k_2) \]
\noindent and so $k_1 = k_2$. \end{proof}

\begin{common-sum} If $N$ is relatively prime to $a$ and $b$, then the sum of the entries in the $i$th column is $\mathcal O$.  If $N$ is relatively prime to $c$ and $d$, then the sum of the entries in the $j$th row is $\mathcal O$.  \end{common-sum}

\begin{proof}  The entries in the $i$th column consist of those $R_k$ corresponding to $k$ such that $x_k = i$.  Similarly, the entries in the $j$th row consist of those $R_k$ corresponding to $k$ such that $y_k = j$.  Hence the sum of the entries in the $i$th column and $j$th row are
\[ \sum_{x_k = i} R_k \qquad \text{and} \qquad \sum_{y_k = j} R_k, \qquad \text{respectively}. \]

First we determine the values of $k$ which occur in the $i$th column.  Since $N$ is relatively prime to $a$ and $b$, there are exactly $N$ pairs $(m,n) \in \left( \mathbb Z / N \mathbb Z \right) \times \left( \mathbb Z / N \mathbb Z \right)$ satisfying $a \, m + b \, n \equiv i - x_1 \pmod{N}$; indeed, given any $m$ we can solve for $n$, and vice-versa.  Hence there are exactly $N$ integers $k \equiv 1 + m + N \, n \pmod{N^2}$ such that $x_k = i$, which we denote by $k_\alpha$.  If we denote $(m_\alpha, n_\alpha) = \phi(k_\alpha)$ using the bijection in \eqref{bijection}, then it is clear we have $\{ \dots, \, m_\alpha, \, \dots \} = \{ \dots, \, n_\alpha, \, \dots \} = \mathbb Z / N \mathbb Z$.

Now we compute the sum of the values in the $i$th column.  Using the group isomorphism in \eqref{isomorphism}, denote $P = \psi \left( (1,0) \right)$ and $Q = \psi \left( (0,1) \right)$ so that we have $R_k = [m] P \oplus [n] Q$ when $(m,n) = \phi(k)$.  This gives the sum
\[ \sum_{x_k = i} R_k = \sum_\alpha R_{k_\alpha} = \sum_\alpha \bigl( [m_\alpha] P \oplus [n_\alpha] Q \bigr) = [m'] P \oplus [n'] Q \]
\noindent where we have set
\[ m' \equiv n' \equiv \sum_\alpha m_\alpha \equiv \sum_\alpha n_\alpha \equiv \sum_{m \in \mathbb Z / N \mathbb Z} m \equiv \frac {N \, (N-1)}2 \pmod{N}. \]
\noindent Since $N$ is assumed odd, this sum is a multiple of $N$ so that $[m'] P = [n'] Q = \mathcal O$.  Hence the sum of the entries in the $i$th column is indeed $\mathcal O$.

A similar argument works for the $j$th row.  \end{proof}

We summarize this as follows.

\begin{magic-square} \label{magic-square} Let $G$ be an abelian group under $\oplus$, and assume that there is a positive odd integer $N$ such that
\[ G[N] = \bigl \{ R_1, \, R_2, \, \dots, \, R_k, \, \dots, \, R_{N^2} \bigr \} \simeq \left( \mathbb Z / N \mathbb Z \right) \times \left( \mathbb Z / N \mathbb Z \right). \]
\noindent Fix integers $a$, $b$, $c$, and $d$ relatively prime to $N$ such that $(a \, d-b \, c)$ is also relatively prime to $N$, and consider the sequence $(x_k, y_k)$ defined by
\[ \begin{aligned} x_k & \equiv x_1 + a \, (k-1) + b \left \lfloor \frac {k-1}N \right \rfloor \pmod {N} \\ y_k & \equiv y_1 + c \, (k-1) + d \left \lfloor \frac {k-1}N \right \rfloor \pmod{N}  \end{aligned} \qquad \text{for} \qquad 1 \leq k \leq N^2. \]

The $N \times N$ grid formed by placing $R_k$ in the $(x_k, y_k)$ position is a magic square, where the sum of each column and each row is the identity $\mathcal O$.  \end{magic-square}

We remark that this method does not exhaust all ways in which a magic square can be generated.  For example, this method does not seem to work for $N$ even.  Indeed, the sum of each column and each row involves the expression $N(N-1)/2$, which in general is not a multiple of $N$.  Also, when $N = 4$, we have the magic square
\[ \begin{tabular}{|c|c|c|c|} \hline 16 & 3 & 2 & 13 \\ \hline 5 & 10 & 11 & 8 \\ \hline 9 & 6 & 7 & 12 \\ \hline 4 & 15 & 14 & 1 \\ \hline \end{tabular}\]

\noindent It is easy to check that such a square cannot be generated by a sequence $(x_k, y_k)$ for any $a$, $b$, $c$, or $d$.  This first appeared in 1514 in an engraving by Albrecht D\"urer entitled ``Melencolia.''

\section{Applications}

We can specialize $a$, $b$, $c$, and $d$ to generate examples of magic squares.

\begin{loubere} Let $G$ be an abelian group under $\oplus$, and assume that there is an odd positive integer $N$ such that
\[ G[N] = \bigl \{ R_1, \, R_2, \, \dots, \, R_k, \, \dots, \, R_{N^2} \bigr \} \simeq \left( \mathbb Z / N \mathbb Z \right) \times \left( \mathbb Z / N \mathbb Z \right). \]
\noindent Then these elements can be placed in an $N \times N$ magic square such that the sum of each column and each row is the identity $\mathcal O$.  \end{loubere}

\begin{proof} We follow the construction using a method first outlined by De la Loub\`ere in 1693.  (An example of how this method works follows at the end of the paper.)  Using Theorem \ref{magic-square}, set $a = 1$, $b = c = -1$, and $d = 2$.  As $N$ is odd, it is relatively prime to these integers as well as the determinant $a \, d - b \, c = 1$. \end{proof}

\emph{Remark.} Theorem \ref{torsion} follows from this corollary, since the group $E[N]$ of $N$-torsion points on an elliptic curve $E$ is isomorphic to $\left( \mathbb Z / N \mathbb Z \right) \times \left( \mathbb Z / N \mathbb Z \right)$.

The following was pointed out to the author by J.-K. Yu.  Upon choosing the basis $\{ P, \, Q \}$ for $G[N]$ given by $P = \psi \bigl( (1,0) \bigr)$ and $Q = \psi \bigl( (0,1) \bigr)$, we may write $R_k = [m] P \oplus [n] Q$ when $(m,n) = \phi(k)$.  (Here, we use the maps defined in \eqref{isomorphism} and \eqref{bijection}.)  In this way, we may identify $R_k$ with $(m,n)$.  If we choose $a = d = 1$ and $b = c = 0$, then we have a magic square upon placing $(m,n) = \phi(k)$ in the $(x_k, y_k)$-position.  In general, if for odd $N$ we have an $N \times N$ Latin Square with the $(m,n)$-position having entry $a_{mn}$ then we may place $(m, a_{mn})$ in the $(x_k, y_k)$-position.  (For more on Latin squares, see \cite{Weis2}.)

We discuss a specific example by considering the 3-torsion on elliptic curves; to this end, set $N = 3$.  We explain how this construction generalizes that in \cite{MR1860304}.  Consider an elliptic curve defined over the complex numbers $\mathbb C$, and let $G = E(\mathbb C)$ be the group of complex points on the curve.  Then it is well-known that we can express the 3-torsion as
\[ E[3] = \bigl \{ A, \ B, \ C, \ D, \ [-1]A, \ [-1]B, \ [-1]C, \ [-1]D, \ \mathcal O \bigr \} \simeq \left( \mathbb Z / 3 \mathbb Z \right) \times \left( \mathbb Z / 3 \mathbb Z \right) \]
\noindent where $B = A \oplus D$ and $[-1]B = C \oplus D$.  If we label these points as
\[ \begin{aligned} R_1 & = \mathcal O, \\ R_2 & = [-1]B, \\ R_3 & = B, \end{aligned} \qquad \begin{aligned} R_4 & = D, \\ R_5 & = [-1]A, \\ R_6 & = [-1]C, \end{aligned} \qquad \begin{aligned} R_7 & = [-1]D, \\ R_8 & = C, \\ R_9 & = A; \end{aligned} \]
\noindent then we can use the magic square from the introduction to place the 3-torsion in a magic square:
\[ \begin{tabular}{|c|c|c|} \hline 3 & 5 & 7 \\ \hline 8 & 1 & 6 \\ \hline 4 & 9 & 2 \\ \hline \end{tabular} \qquad \implies \qquad \begin{tabular}{|r|r|r|} \hline $B$ & $[-1]A$ & $[-1]D$ \\ \hline $C$ & $\mathcal O$ & $[-1]C$ \\ \hline $D$ & $A$ & $[-1]B$ \\ \hline \end{tabular} \]
\noindent We can also compute this magic square using the method in the proof of the corollary.   Choosing the basis $P = [-1]B$ and $Q = D$; it can be easily checked that $R_k = [m]P \oplus [n]Q$ when $(m,n) = \phi(k)$.  If we also choose $(x_1, y_1) = (2,2)$ as the center of the $3 \times 3$ grid, then $R_k$ may be placed in the $(x_k, y_k)$-position, where
\[ \begin{aligned} x_k & \equiv x_1 + (k-1) - \left \lfloor \frac {k-1}N \right \rfloor \pmod {N} \\ y_k & \equiv y_1 - (k-1) + 2 \left \lfloor \frac {k-1}N \right \rfloor \pmod{N}  \end{aligned} \qquad \text{for} \qquad 1 \leq k \leq N^2. \]
\noindent As mentioned before, this is known as De la Loub\`ere's method or the Siamese method.

\bibliographystyle{plain}

\end{document}